\documentclass{amsart}
\usepackage{amsmath}
\usepackage{amssymb}
\usepackage{amsthm}
\usepackage{tikz}
\usepackage{float}
\usepackage{braket}
\usepackage{enumitem}
\usepackage[colorlinks = true, citecolor = blue]{hyperref}
\usepackage{marvosym}
\usepackage[margin=1in]{geometry}
\usepackage{multirow}

\title{Bounding the number of arithmetical structures on graphs}
\author{Christopher Keyes}
\address{Department of Mathematics, Emory University, Atlanta, GA 30322}
\email{\href{mailto:christopher.keyes@emory.edu}{christopher.keyes@emory.edu}}
\email{\href{mailto:tomer.reiter@emory.edu}{tomer.reiter@emory.edu}}
\author{Tomer Reiter}
\date{\today}

\newtheorem{theorem}{Theorem}[section]
\newtheorem{lemma}[theorem]{Lemma}

\newtheorem{corollary}[theorem]{Corollary}
\theoremstyle{definition}
\newtheorem{example}[theorem]{Example}

\newtheorem{construction}[theorem]{Construction}
\newtheorem{remark}[theorem]{Remark}
\numberwithin{equation}{section}
\numberwithin{figure}{section}
\numberwithin{table}{section}

\newcommand{\N}{\mathbb{N}}
\newcommand{\Z}{\mathbb{Z}}

\newcommand{\wt}[1]{\widetilde{#1}}
\newcommand{\wh}[1]{\widehat{#1}}
\renewcommand{\r}{\textbf{r}}
\renewcommand{\d}{\textbf{d}}
\newcommand{\E}[1]{\#E(#1)}

\DeclareMathOperator{\dec}{dec}
\DeclareMathOperator{\diag}{diag}

\begin{document}

\begin{abstract}
	Let $G$ be a connected undirected graph on $n$ vertices with no loops but possibly multiedges. Given an arithmetical structure $(\r, \d)$ on $G$, we describe a construction which associates to it a graph $G'$ on $n-1$ vertices and an arithmetical structure $(\r', \d')$ on $G'$. By iterating this construction, we derive an upper bound for the number of arithmetical structures on $G$ depending only on the number of vertices and edges of $G$. In the specific case of complete graphs, possibly with multiple edges, we refine and compare our upper bounds to those arising from counting unit fraction representations.
\end{abstract}

\maketitle

\section{Introduction}

Let $G$ be a connected undirected graph with $n$ vertices labeled $v_1, \ldots, v_n$, containing no loops but possibly multiedges. Throughout this paper, we use $E(G)$ to refer to the edge set of $G$, $\delta_{ij}$ to denote the number of edges between $v_i$ and $v_j$, and $\deg v$ for the degree of the vertex $v$. An \textit{arithmetical structure on $G$} is a pair $(\textbf{r}, \textbf{d}) \in \N^{n} \times \N^n$, such that $\gcd(\textbf{r}) = \gcd(r_1, ..., r_n) = 1$, satisfying the system
\begin{align}\label{eq:astruct_system}
	\nonumber r_1d_1 = r_2\delta_{12} + &\cdots + r_n\delta_{1n}\\
	 r_2d_2 = r_1\delta_{21} + &\cdots + r_n\delta_{2n}\\
	\nonumber   &\vdots\\
	\nonumber r_nd_n = r_1\delta_{n1} + &\cdots + r_{n-1}\delta_{n(n-1)}.
\end{align}
Equivalently, an arithmetical structure is the data of $\textbf{r},\textbf{d} \in \N^n$ satisfying the matrix equation
\begin{equation}\label{eq:astruct_matrix}
\begin{pmatrix} -d_1 &\delta_{12}& \cdots &\delta_{1n}\\
				\delta_{21} & -d_2 & \cdots &\delta_{2n}\\
				\vdots & \vdots & \ddots & \vdots\\
				\delta_{n1} & \delta_{n2} & \cdots &-d_n
				\end{pmatrix}
	\begin{pmatrix} r_1\\ r_2\\ \vdots \\ r_n \end{pmatrix} = 
	\begin{pmatrix} 0\\ 0\\ \vdots \\ 0 \end{pmatrix}.
\end{equation}
Note that specifying $\textbf{r}$ such that $r_i \mid \sum_{j \neq i} r_j\delta_{ij}$ is sufficient to recover $\textbf{d}$. Thus we may simply refer to $\textbf{r}$ as an arithmetical structure on $G$. We use $A(G)$ to denote the set of arithmetical structures on a graph $G$.

We remark that we could extend this definition of an arithmetical structure to a graph with loops. We simply amend \eqref{eq:astruct_system} by requiring 
\[r_id_i = \sum_{j=1}^n r_j\delta_{ij}\]
for all $i$. However, by absorbing $\delta_{ii}$ into $d_i$ for each $i$, it can be seen that $\textbf{r}$ defines an arithmetical structure on $G_0$, where $G_0$ is the graph obtained by removing all loops from $G$. Thus $A(G)$ is in one-to-one correspondence with $A(G_0)$, and for the remainder of this paper we will assume $G$ contains no loops.

While combinatorial in nature, arithmetical structures are related to the study of special fibers of relative proper minimal models of curves. They were introduced by Lorenzini, who proved that $A(G)$ is finite \cite{Lorenzini}. Aside from certain special cases, little is known beyond finiteness about $\#A(G)$. Braun et.\ al.\ \cite{GlassEtAl} succeeded in enumerating the number of arithmetical structures when $G$ is a path or a cycle, where they found connections to the Catalan numbers and certain binomial coefficients. Archer et.\ al\ \cite{ArcherEtAl} considered bidents --- paths with two prongs at one end --- and gave bounds again in terms of the Catalan numbers. Glass and Wagner \cite{GlassWagner} studied arithmetical structures on paths with a doubled edge, and formulated a conjecture for how $\#A(G)$ grows in this case, depending on the path length and the location of the doubled edge.

In this paper, we introduce a construction in Section \ref{sec:recursive} to reduce an arithmetical structure on a graph $G$ with $n$ vertices into an arithmetical structure on an associated graph $G'$ with $n-1$ vertices. Our primary application of this construction is to derive an explicit general upper bound for the number of arithmetical structures on a graph $G$, depending only on the number of vertices and edges.
\begin{theorem}\label{thm:bound}
Let $G$ be a connected, undirected graph on $n$ vertices, with no loops but possible multiedges. Then the following is an upper bound for the number of arithmetical structures on $G$.
\[\#A(G) \leq \frac{n!}{2} \cdot \E{G}^{2^{n-2} - 1} \cdot \E{G}^{2^{n-1} \cdot \frac{1.538\log(2)}{(n-1)\log(2) + \log(\log(\E{G}))}}.\]
\end{theorem}
\noindent Section \ref{sec:bounds} is devoted to the proof of Theorem \ref{thm:bound}.

Our construction generalizes the \textit{smoothing} process used in \cite{GlassEtAl}, \cite{ArcherEtAl}, and \cite{GlassWagner}. In certain special cases, it is the inverse of Lorenzini's \textit{blowup} construction \cite[1.8]{Lorenzini}, and extends observations made by Corrales and Valencia about the arithmetical structures on the \textit{clique-star} transform of a graph \cite{CorralesValencia_clique_star}.

In Section \ref{sec:Egyptian}, we discuss the special case of graphs with $n$ vertices and $m$ edges between each pair of vertices, which we denote by $mK_n$. We first give a refinement of Theorem \ref{thm:bound} before making connections between their arithmetical structures and Egyptian fractions. An \textit{Egyptian fraction} describes an integer fraction $a/m$ as the sum of unit fractions,
\begin{equation}\label{eq:egypt_frac_general}
	\frac{1}{x_1} + \cdots + \frac{1}{x_n} = \frac{a}{m}.
\end{equation}
These representations have been studied from many angles over the years --- for a brief survey see the introduction of \cite{Bleicher}. There also remain many open problems about Egyptian fractions, including the Erd{\"o}s--Straus Conjecture, which concerns the existence of a representation for all $m$ in the case where $a=4$ and $n=3$ in \eqref{eq:egypt_frac_general}. See \cite{Guy} for more open problems related to Egyptian fractions.

We are interested in Egyptian fractions with $a=1$,
\begin{equation}\label{eq:egypt_frac}
	\frac{1}{x_1} + \cdots + \frac{1}{x_n} = \frac{1}{m},
\end{equation} In Theorem \ref{thm:egypt_complete} we describe a one-to-one correspondence between integer solutions to \eqref{eq:egypt_frac} and $A(mK_n)$. This connection in the case of $K_n$ was also noted in \cite{HarrisLouwsma}, in which the integers that can appear as the largest $r$-value for an arithmetical structure on $K_n$ were partially classified. We may then use the known results about Egyptian fraction representations to give an asymptotic upper bound for $\#A(mK_n)$ which improves on Theorem \ref{thm:bound}. 

\subsection*{Acknowledgements} The authors would like to thank Nathan Kaplan and Joel Louwsma for their comments which helped strengthen this paper and alerted them to several results in the literature they had not found. They would also like to thank David Zureick--Brown for many helpful conversations and for suggesting edits on a draft of this paper.

\section{A recursive construction}
\label{sec:recursive}

We now describe a construction which associates to an arithmetical structure $(\r,\d)$ on $G$ an arithmetical structure $(\r', \d')$ on an associated graph $G'$ possessing $n-1$ vertices. The process of obtaining $G'$ is described precisely below in Construction \ref{con:recursive}.

\begin{construction}
\label{con:recursive}
	Let $G$ be a connected undirected graph with $n$ vertices, with $v_i$ and $\delta_{ij}$ having their usual meanings. For any choice of vertex $v_i$ and positive ingeter $s$, we define a graph $G(v_i,s)$ as follows: $G(v_i, s)$ has $n-1$ vertices, obtained by removing the $i$-th vertex from $G$, so
	\[V\left(G(v_i,s)\right) = V(G) - \set{v_i}.\]
	The edges of $G(v_i,s)$ are given by
	\[\delta_{jk}' = \delta_{ij}\delta_{ik} + s\delta_{jk},\]
	where $\delta_{jk}'$ denotes the number of edges between distinct vertices $v_j$ and $v_k$ in $G(v_i,s)$.
\end{construction}

\begin{remark} \label{rem:clique_star}
	Alternatively, we could envision $G(v_i, s)$ as the union of $s(G - v_i)$ with the graph obtained by performing a star-clique type operation around $v_i$, in which $v_i$ is removed from its star and $\delta_{ij}\delta_{ik}$ edges are added between each pair of distinct vertices $v_j$ and $v_k$. This description makes more apparent that when $s=1$ and the star of $v_i$ is simple (i.e.\ $\delta_{ij} = 1$ for all $v_j$ adjacent to $v_i$), Construction \ref{con:recursive} is inverse to the clique-star transform described in \cite[\S 5]{CorralesValencia_clique_star}. More precisely, using the notation of \cite[\S 5]{CorralesValencia_clique_star}, suppose $G$ is a graph containing a clique $C$ and $\wt{G} = cs(G,C)$ is its clique star transform with new vertex $v$. Then our construction applied to $v$ with $s = 1$ recovers the original graph, i.e.\ $\wt{G}(v, 1) = G$.
\end{remark}

We illustrate Construction \ref{con:recursive} with an example.

\begin{example}
\label{ex:recursive}
Consider the graph $G$ shown below in Figure \ref{fig:ex_G_1}. 

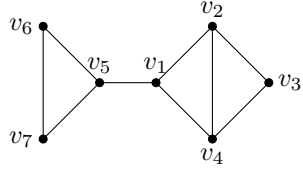
\begin{figure}[H]
\centering
\caption{$G$ with vertices labeled.}
\label{fig:ex_G_1}
\vspace{1ex}

\begin{tikzpicture}[scale=0.75]
	\filldraw[black] (2, 0) circle (2pt) node[right] {$v_3$};
	\filldraw[black] (1,1) circle (2pt) node[above] {$v_2$};
	\filldraw[black] (1,-1) circle (2pt) node[below] {$v_4$};
	\filldraw[black] (0, 0) circle (2pt) node[above] {$v_1$};
	\filldraw[black] (-1, 0) circle (2pt) node[above] {$v_5$};
	\filldraw[black] (-2, 1) circle (2pt) node[left] {$v_6$};
	\filldraw[black] (-2, -1) circle (2pt) node[left] {$v_7$};
	
	\draw[black] (0,0) -- (1,1);
	\draw[black] (1,1) -- (2,0);
	\draw[black] (1,1) -- (1,-1);
	\draw[black] (2,0) -- (1,-1);
	\draw[black] (1,-1) -- (0,0);
	\draw[black] (-1,0) -- (0,0);
	\draw[black] (-1,0) -- (-2,1);
	\draw[black] (-1,0) -- (-2,-1);
	\draw[black] (-2,1) -- (-2,-1);
\end{tikzpicture}
\end{figure}

\noindent Using Construction \ref{con:recursive} with $i=1$ and $s = 2$, we obtain $G' = G(v_1,2)$. This is shown step-by-step in Figure \ref{fig:ex_G_2}. In step (i) we highlight $v_1$ and its incident edges in red to be removed. Step (ii) shows the graph $2(G - v_1)$ and finally step (iii) shows in blue the additional $\delta_{1j}\delta_{1k}$ edges added for each pair of remaining vertices.

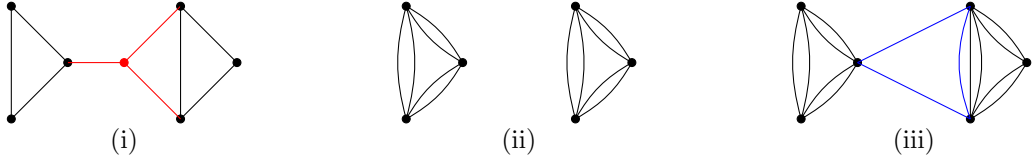
\begin{figure}[H]
\centering
\caption{Obtaining $G' = G(v_1, 2)$ from Construction \ref{con:recursive}.}
\label{fig:ex_G_2}
\vspace{1ex}

\begin{tikzpicture}[scale=0.75]
	\filldraw[black] (2, 0) circle (2pt);
	\filldraw[black] (1,1) circle (2pt);
	\filldraw[black] (1,-1) circle (2pt);
	\filldraw[red] (0, 0) circle (2pt);
	\filldraw[black] (-1, 0) circle (2pt);
	\filldraw[black] (-2, 1) circle (2pt);
	\filldraw[black] (-2, -1) circle (2pt);
	
	\draw[red] (0,0) -- (1,1);
	\draw[black] (1,1) -- (2,0);
	\draw[black] (1,1) -- (1,-1);
	\draw[black] (2,0) -- (1,-1);
	\draw[red] (1,-1) -- (0,0);
	\draw[red] (-1,0) -- (0,0);
	\draw[black] (-1,0) -- (-2,1);
	\draw[black] (-1,0) -- (-2,-1);
	\draw[black] (-2,1) -- (-2,-1);
	
	\draw (0,-1) node[below] {(i)};
	
	\filldraw[black] (9, 0) circle (2pt);
	\filldraw[black] (8,1) circle (2pt);
	\filldraw[black] (8,-1) circle (2pt);
	\filldraw[black] (6, 0) circle (2pt);
	\filldraw[black] (5, 1) circle (2pt);
	\filldraw[black] (5, -1) circle (2pt);
	
	\draw[black] (8,1) to [bend right = 15] (9,0);
	\draw[black] (8,1) to [bend left = 15] (9,0);
	\draw[black] (8,1) to [bend right = 15] (8,-1);
	\draw[black] (8,1) to [bend left = 15] (8,-1);
	\draw[black] (9,0) to [bend right = 15] (8,-1);
	\draw[black] (9,0) to [bend left = 15] (8,-1);
	\draw[black] (6,0) to [bend right = 15] (5,1);
	\draw[black] (6,0) to [bend left = 15] (5,1);
	\draw[black] (6,0) to [bend right = 15] (5,-1);
	\draw[black] (6,0) to [bend left = 15] (5,-1);
	\draw[black] (5,1) to [bend right = 15] (5,-1);
	\draw[black] (5,1) to [bend left = 15] (5,-1);
	
	\draw (7,-1) node[below] {(ii)};
	
	\filldraw[black] (16, 0) circle (2pt);
	\filldraw[black] (15,1) circle (2pt);
	\filldraw[black] (15,-1) circle (2pt);
	\filldraw[black] (13, 0) circle (2pt);
	\filldraw[black] (12, 1) circle (2pt);
	\filldraw[black] (12, -1) circle (2pt);
	
	\draw[black] (15,1) to [bend right = 15] (16,0);
	\draw[black] (15,1) to [bend left = 15] (16,0);
	\draw[black] (15,1) to (15,-1);
	\draw[black] (15,1) to [bend left = 20] (15,-1);
	\draw[black] (16,0) to [bend right = 15] (15,-1);
	\draw[black] (16,0) to [bend left = 15] (15,-1);
	\draw[black] (13,0) to [bend right = 15] (12,1);
	\draw[black] (13,0) to [bend left = 15] (12,1);
	\draw[black] (13,0) to [bend right = 15] (12,-1);
	\draw[black] (13,0) to [bend left = 15] (12,-1);
	\draw[black] (12,1) to [bend right = 15] (12,-1);
	\draw[black] (12,1) to [bend left = 15] (12,-1);
	\draw[blue] (13,0) -- (15,1);
	\draw[blue] (13,0) -- (15,-1);
	\draw[blue] (15,1) to [bend right = 20] (15,-1);
	
	\draw (14,-1) node[below] {(iii)};
\end{tikzpicture}
\end{figure}

Let $\r = (2,1,1,2,1,1,1)$, which gives an arithmetical structure on $G$ with $\textbf{d} = (2,5,3,2,4,2,2)$. We have $s = d_1 = 2$, and one may check that $\r' = (1,1,2,1,1,1)$ gives an arithmetical structure on $G' = G(v_1,2)$. This turns out to be an example of a more general phenomenon --- for any arithmetical structure $\r$ on $G$, take $\r' = (r_1, \ldots, \widehat{r_i}, \ldots, r_n)$, where $\widehat{r_i}$ denotes removal of the $i$-th entry from the tuple. Then for the graph $G(v_i,d_i)$, we find $\r'$ satisfies the requirements of \eqref{eq:astruct_system} for some appropriate $\d'$. Hence, it is an arithmetical structure after possible scaling. Lemma \ref{lem:recursive} verifies this observation in general.

\end{example}

\begin{lemma}\label{lem:recursive}
	Fix an arithmetical structure $(\r,\d)$ on $G$ and a vertex $v_i$ of $G$, and let $G' = G(v_i,d_i)$ as given by Construction \ref{con:recursive}. Set $g = \gcd(r_1, \ldots, \wh{r_i}, \ldots, r_n)$ and $\r' = (r_1/g, \ldots, \wh{r_i/g}, \ldots, r_n/g)$. Then $\r'$ is an arithmetical structure on $G'$.
\end{lemma}

\begin{proof}
It suffices to consider the case $i=1$, as we can always renumber the vertices of $G$ so that $v_1$ is removed. By \eqref{eq:astruct_system}, we have the system
\begin{align}\label{eq:astruct_lemma}
	                  \nonumber r_2d_2 &= \frac{r_2\delta_{12} + \cdots + r_n\delta_{1n}}{d_1}\delta_{21} + \cdots + r_n\delta_{2n}\\
	 & \vdots\\
	\nonumber r_nd_n &= \frac{r_2\delta_{12} + \cdots + r_n\delta_{1n}}{d_1}\delta_{n1} + \cdots + r_{n-1}\delta_{n(n-1)}.
\end{align}
So for $2 \leq i \leq n$ we have:
\begin{align}\label{eq:astruct_on_Gprime}
	\nonumber r_i(d_1d_i - \delta_{1i}^2) &= \delta_{1i}(r_2\delta_{12} + \cdots + r_n\delta_{1n}) + d_1\sum_{2 \leq j \leq n, i\neq j} r_j\delta_{ij}\\
	 &= \sum_{2 \leq j \leq n, i \neq j} r_j(\delta_{i1}\delta_{1j} + d_1\delta_{ij}).
\end{align} 
Notice that Construction \ref{con:recursive} gives $\delta_{ij}' = \delta_{i1}\delta_{1j} + d_1\delta_{ij}$, so we have 
\[r_i(d_1d_i - \delta_{1i}^2) = \sum_{2 \leq j \leq n, i \neq j} r_j \delta_{ij}'\]
for all $2 \leq i \leq n$, which is precisely \eqref{eq:astruct_system} on the new graph $G'$. If $g > 1$ then we need to divide each $r_i$ by $g$ to obtain another arithmetical structure, which corresponds to scaling \eqref{eq:astruct_on_Gprime} by $1/g$. Let $(r_2', r_3', \ldots, r_n'), (d_2', d_3', \ldots, d_n')$ denote the new arithmetical structure on $G'$. Then explicitly, $r_i' = r_i/g$ and $d_i' = d_1d_i - \delta_{1i}^2$. Since the $r_i'$ are positive integers, the new numbers of edges between pairs of vertices are non-negative, and $G'$ is clearly connected, we have that $d_i'$ are also positive integers. Thus $(\textbf{r}', \textbf{d}')$ is an arithmetical structure on $G'$. \end{proof}

\begin{remark}
	If $(\r, \d)$ is an arithmetical structure on $G$, $d_1 = 1$, and $G' = G(v_1,1)$, then $G$ with its arithmetical structure $(\r,\d)$ is the blow up \cite[1.8]{Lorenzini} of $G'$ with its arithmetical structure $(\r', \d')$. In this case, one translates between our construction and Lorenzini's by taking $M = D' - A'$, where $D' = \diag(\d')$, $A'$ is the adjacency matrix of $G'$, and $\textbf{q}^T = (\delta_{12}, \ldots, \delta_{1n})$. Then $M_{\textbf{q}} = D - A$ is the matrix corresponding to the original arithmetical structure on $G$. As a consequence, we observe that when $d_1 = 1$, the critical groups of the arithmetical structures $\r$ on $G$ and $\r'$ on $G'$ are isomorphic. It may be interesting to study the relationships between the critical groups $(\r, \d)$ on $G$ and $(\r', \d')$ on $G(v_i,d_i)$ more generally.
\end{remark}

In this paper, when we have a fixed arithmetical structure $(\r,\d)$ on $G$ we will only be interested in the case where $s = d_i$ coming from Lemma \ref{lem:recursive}, and as mentioned in the proof we can always renumber the vertices of $G$ such that $i=1$. Hence, for the remainder of this paper we simply take $G' = G(v_1,d_1)$ when it will not create confusion. We will occasionally make use of the more general construction, as it is needed for the proof of Theorem \ref{thm:bound}.

\begin{example}[Paths]
\label{ex:paths}
	Let $P_n$ denote the path with $n$ vertices, i.e.\ $\delta_{ij} = 1$ if $j = i \pm 1$ and $\delta_{ij} = 0$ otherwise. Arithmetical structures on paths have been studied extensively, and it has been shown that $\#A(P_n) = C_{n-1} = \frac{1}{n} \binom{2n-2}{n-1}$, where $C_{n-1}$ denotes the $(n-1)$-th Catalan number \cite[Theorem 3]{GlassEtAl}. 
	
	If $n \geq 3$, we may apply Construction \ref{con:recursive} at vertex $i$ with $1 < i < n$ and find $P_n(v_i,1) = P_{n-1}$. To see this, we check that for $j<k$, as long as $(j,k) \neq (i-1,i+1)$, we have $\delta_{jk}' = \delta_{jk}$, since one of $\delta_{ij}$ or $\delta_{ik}$ is 0. Then we have \[\delta_{(i-1)(i+1)}' = \delta_{(i-1)i}\delta_{i(i+1)} + 1\cdot \delta_{(i-1)(i+1)} = 1.\]
	In particular, given an arithmetical structure $(\r,\d)$ on $P_n$ with $d_i = 1$ for some $1 < i < n$, we obtain an arithmetical structure $\r' = (r_1, \ldots, \wh{r_i}, \ldots, r_n)$ on $P_{n-1}$ ($r_1 = r_n = 1$ so we have automatically have $g=1$). This is precisely the \textit{smoothing} process of \cite[Proposition 5]{GlassEtAl}, so one may view Construction \ref{con:recursive} and Lemma \ref{lem:recursive} as a more general version of the smoothing of a path.
\end{example}

We conclude this section with another illustrative example which we will study in greater depth in Section \ref{sec:Egyptian}.

\begin{example}[complete (multi)graphs]
\label{ex:mKn}
	Let $K_n$ denote the complete graph on $n$ vertices. We will use $mK_n$ to denote the complete graph $K_n$ but instead with $m$ edges between each two vertices. The regular nature of this graph allows for a concise description of the graph $mK_n(v_1, s)$ obtained from Construction \ref{con:recursive} on $mK_n$.

	After removing the vertex $v_1$ and all incident edges, we are left with $n-1$ vertices. The value of $\delta_{ij}'$ is given by $\delta_{ij}' = \delta_{1i}\delta_{1j} + s\delta_{ij} = m^2 + sm$. Thus, $mK_n(v_1, s) = (m^2 + sm)K_{n-1}$.
	
	We illustrate this below with the arithmetical structure $\r = (6,3,2,1)$ on $K_4$, which gets reduced to the arithmetical structure $\r' = (3,2,1)$ on $K_4(v_1,1) = 2K_3$, which is further reduced to $\r'' = (2,1)$ on $2K_3(v_2,2) = 8K_2$.

\begin{figure}[H]
\centering
\caption{Applying Construction \ref{con:recursive} twice to $K_4$. Vertices are labeled with their $(r_i,d_i)$ values.}
\label{fig:ex_K_4}
\vspace{1ex}
	
	\begin{tikzpicture}[scale=0.75]
	\filldraw[black] (0, -1) circle (2pt) node[below] {$(3,3)$};
	\filldraw[black] (-1, 0) circle (2pt) node[left] {$(1,11)$};
	\filldraw[black] (1, 0) circle (2pt) node[right] {$(2,5)$};
	\filldraw[red] (0, 1) circle (2pt) node[above] {$(6,1)$};
	
	\draw[red] (0,1) -- (1,0);
	\draw[red] (0,1) -- (-1,0);
	\draw[red] (0,1) -- (0,-1);
	\draw[black] (-1,0) -- (1,0);
	\draw[black] (-1,0) -- (0,-1);
	\draw[black] (1,0) -- (0,-1);
	
	\draw[black, ->] (3,0) -- (4,0);
	
	\filldraw[blue] (7, -1) circle (2pt) node[below] {$(3,2)$};
	\filldraw[black] (6, 0) circle (2pt) node[left] {$(1,10)$};
	\filldraw[black] (8, 0) circle (2pt) node[right] {$(2,4)$};
	
	\draw[blue] (7,-1) to [bend right = 15] (6,0);
	\draw[blue] (7,-1) to [bend left = 15] (6,0);
	\draw[blue] (7,-1) to [bend right = 15] (8,0);
	\draw[blue] (7,-1) to [bend left = 15] (8,0);
	\draw[black] (8,0) to [bend right = 15] (6,0);
	\draw[black] (8,0) to [bend left = 15] (6,0);
	
	\draw[black, ->] (10,0) -- (11,0);
	
	\filldraw[black] (13, 0) circle (2pt) node[left] {$(1,16)$};
	\filldraw[black] (15, 0) circle (2pt) node[right] {$(2,4)$};
	\draw[black] (13,0) to [bend right = 10] (15,0);
	\draw[black] (13,0) to [bend right = 25] (15,0);
	\draw[black] (13,0) to [bend right = 40] (15,0);
	\draw[black] (13,0) to [bend right = 60] (15,0);
	\draw[black] (13,0) to [bend left = 10] (15,0);
	\draw[black] (13,0) to [bend left = 25] (15,0);
	\draw[black] (13,0) to [bend left = 40] (15,0);
	\draw[black] (13,0) to [bend left = 60] (15,0);

\end{tikzpicture}
\end{figure}
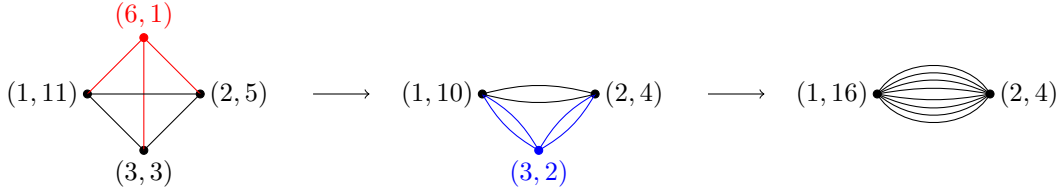
\end{example}

\section{Upper bounds on $\#A(G)$}
\label{sec:bounds}

In this section, we leverage Construction \ref{con:recursive} inductively to derive an upper bound for the number of arithmetical structures on an arbitrary graph. To state our main result, recall the divisor function, which counts the number of positive divisors of an integer $n$, denoted here by $\sigma_0(n)$. This theorem is a strengthening of Theorem \ref{thm:bound}, so we will prove this instead.

\begin{theorem}\label{thm:bound_restated}
	Let $G$ be a connected, undirected graph on $n$ vertices, with no loops but possible multiedges. Suppose $f$ is any monotonically increasing function such that $\sigma_0(m) \leq f(m)$ for all positive integers $m$. Then
\begin{equation} \label{eq:upper_bound_f}
\#A(G) \leq \frac{n!}{2} \E{G}^{2^{n-2} - 1} \cdot f\left(\E{G}^{2^{n-1}}\right).
\end{equation}
\end{theorem}

We will now justify that Theorem \ref{thm:bound} follows from Theorem \ref{thm:bound_restated}. In \cite{Nicolas}, an explicit upper bound for the divisor function is given to be
\[\sigma_0(m) \leq m^{\frac{1.538\log(2)}{\log(\log(m))}}.\]
Note that the right hand side is monotonically increasing in $m$. Taking $f(m) = m^{\frac{1.538\log(2)}{\log(\log(m))}}$ produces the upper bound for $\#A(G)$ given by Theorem \ref{thm:bound} in the introduction, so it follows directly from Theorem \ref{thm:bound_restated}.

The proof of Theorem \ref{thm:bound_restated} proceeds by induction on the number of vertices of $G$. We take care of the base case, when $n=2$, in Lemma \ref{lem:n=2}. This is where the divisor function is introduced. We then prove an independent result in Theorem \ref{thm:largest_r1}, which provides an upper bound for the $r_i$-values depending only on $n$ and $\E{G}$. Next we prove Theorem \ref{thm:bound_restated} using Lemma \ref{lem:n=2} and some of the ideas from the proof of Theorem \ref{thm:largest_r1}. We conclude this section by comparing our result to the known values of $\#A(G)$ when $G = P_{n+1}$ is a path.

\begin{lemma}\label{lem:n=2}
Let $G$ be a graph with two vertices, $v_1$ and $v_2$. If $(r_1, r_2)$ is an arithmetical structure on $G$ then $r_1, r_2 \mid \E{G}$ and so $r_1,r_2 \leq \E{G}$. The total number of arithmetical structures on $G$ is precisely $\sigma_0(\E{G}^2)$. 
\end{lemma}
\begin{proof}
The divisibility statement follows from the fact that $r_1 \mid \E{G}r_2$ and $\gcd(r_1, r_2) = 1$. For the second part, we provide a bijection between the set of arithmetical structures on $G$ and divisors of $\E{G}^2$, defined by sending
$$(r_1, r_2) \mapsto \frac{\E{G}}{r_1}r_2$$
This is clearly well-defined. If $(r_1, r_2)$ and $(r_1', r_2')$ get mapped to the same integer then $r_1'r_2 = r_1r_2'$ and since $\gcd(r_1, r_2) = \gcd(r_1', r_2') = 1$ we get $r_1\mid r_1'$ and $r_1' \mid r_1$. Thus $r_1 = r_1'$ and $r_2 = r_2'$, so this map is injective. 

To demonstrate surjectivity, let $\E{G} = p_1^{\alpha_1}\cdots p_k^{\alpha_k}$, and let $p_1^{\beta_1} \cdots p_k^{\beta_k}$ be a factor of $\E{G}^2$. Then for each $i$, if $\beta_i \leq \alpha_i$, we add a factor of $p_i^{\alpha_i - \beta_i}$ to $r_1$, and otherwise we add a factor of $p_i^{\beta_i - \alpha_i}$ to $r_2$. This will result in the power of $p_i$ in the image of $(r_1,r_2)$ being $\beta_i$, and hence the $(r_1, r_2)$ obtained by this procedure has image equal to $p_1^{\beta_1} \cdots p_k^{\beta_k}$. \end{proof}

\begin{remark}
\label{rem:A_vs_Adec}
	We consider order to matter when enumerating $A(G)$. For example, if $G$ is the graph with two vertices and three edges ($G = 3K_2$ in the notation of Example \ref{ex:mKn}), we count the arithmetical structures $\r =(1,3)$ and $\r =(3,1)$ separately. In this case, $\#A(G) = \sigma_0(3^2) = 3$.
\end{remark}

With Construction \ref{con:recursive} and Lemma \ref{lem:n=2} we can now give an upper bound for the largest possible $r_1$ value on a given graph $G$, which depends only on the number of vertices and edges. We will then prove Theorem \ref{thm:bound_restated}.

\begin{theorem}\label{thm:largest_r1}
Let $\textbf{r}$ be an arithmetical structure on a graph $G$ with $n$ vertices. Reorder the vertices so that $r_1 \geq r_2 \ldots \geq r_n$. Then 
$$r_1 \leq \frac{1}{(n-1)!} \cdot \E{G}^{3\cdot2^{n - 2} -2}$$
\end{theorem}

\begin{proof}
The case $n = 2$ reduces to $r_1 \leq \#E(G)$, which follows from Lemma \ref{lem:n=2}, so we assume the statement is true for all graphs with $n - 1$ vertices. Let $G$ be a graph with $n$ vertices and take $(r_1, r_2, \ldots, r_n)$ to be an arithmetical structure on $G$.

By \eqref{eq:astruct_system}, we have
$$r_1d_1 = r_2\delta_{12} + \cdots + r_n\delta_{1n}$$
and so
$$r_1 \leq \left(\sum_{i=2}^n \delta_{1i}\right) r_2 = \deg(v_1) r_2.$$
Since $(r_2/g, \ldots, r_n/g)$ is an arithmetical structure on $G(v_1, d_1)$, where $g$ and $G'$ are as in Lemma \ref{lem:recursive}, we in turn have the inequality
$$r_1 \leq \deg(v_1) \cdot \left(\max_{(\textbf{r}', \textbf{d}') \in A(G)} \left(\max_{\textbf{r}'' \in A(G(v_1, d_1'))} r_2''\right)\right) \cdot \left(\max_{\textbf{r}' \in A(G)} \gcd(\textbf{r}')\right).$$
Here the nested maximum is over all possible arithmetical structures $(r_1', r_2', \ldots, r_n', d_1', d_2', \ldots, d_n') \in A(G)$, and over all possible arithmetical structure $(r_2'', r_3'', \ldots, r_n'') \in A(G(v_1, d_1'))$. The second maximum is over all arithmetical structures $\textbf{r}' \in A(G)$. 

Since $g \mid r_2, \ldots, r_n$ and $\gcd(r_1, r_2, \ldots, r_n) = 1$, we have $\gcd(r_1, g) = 1$. Then by \eqref{eq:astruct_system}, we have for all $i > 1$,
\begin{align} \label{eq:g_bound}
\nonumber &r_id_i = \sum_{j\neq i} r_j\delta_{ij} \text{, and so} \\
&r_1\delta_{1i} = r_id_i - \sum_{j > 1, j \neq i} r_j \delta_{ij}.
\end{align}
Since $\gcd(g, r_1) = 1$, and $g$ divides the right hand side of Equation \ref{eq:g_bound}, we have $g \mid \delta_{1i}$, so $g \leq \delta_{1i}$. Summing over all $i \neq 1$, we get 
\[(n-1)g \leq \sum_{i = 2}^n \delta_{in} = d(v_1) \leq \E{G},\]
so $g \leq \E{G}/(n-1)$. Here we trivially bound $\deg(v_1) \leq \E{G}$, and in general we can't do any better, since all edges in the graph could be incident to $v_1$.

By the inductive hypothesis we now have
\begin{align} \label{eq:largest_r1}
r_1 \leq \frac{\E{G}^2}{n-1} \cdot \left(\max_{(\textbf{r}', \textbf{d}') \in A(G)} \left(\max_{\textbf{r}'' \in A(G(v_1, d_1'))} r_2''\right)\right) \leq \frac{\E{G}^2}{(n-1)!} \left(\max_{(\textbf{r}', \textbf{d}') \in A(G)} \E{G(v_1, d_1')}\right)^{3\cdot2^{n-2}-2}
\end{align}

We will give an upper bound for this maximum. Using Construction \ref{con:recursive}, we have
$$\E{G(v_1, d_1')} \leq d_1' e_{\neq 1} + \binom{\deg(v_1)}{2}$$
where $e_{\neq 1}$ is the number of edges on $G$ not incident to $v_1$. The binomial coefficient arises from the final step of the construction for $G(v_1, d_1')$, where at worst every pair of edges adjacent to $v_1$ will add a new edge in $G(v_1, d_1')$. Now, we have again by \eqref{eq:astruct_system}
$$r_1'd_1' = r_2'\delta_{12} + ... + r_n'\delta_{1n} \leq \deg(v_1)r_2'$$
and so $d_1' \leq \deg(v_1)$. Therefore we have
$$\E{G(v_1, d_1')} \leq \deg(v_1)(\E{G} - \deg(v_1)) +  \binom{\deg(v_1)}{2}$$
which is a quadratic function in $\deg(v_1)$, which has a maximum of $(2\E{G} + 1)^2/8 \leq \E{G}^2$, so \\$\E{G(v_1, d_1')} \leq \E{G}^2$. So \ref{eq:largest_r1} becomes
$$r_1 \leq \frac{\E{G}^2}{n-1} \frac{1}{(n-2)!} \cdot \E{G}^{3\cdot2^{n-1}-4} = \frac{1}{(n-1)!}\E{G}^{3\cdot2^{n-1} - 2}.$$
\end{proof}

\begin{proof}[Proof of Theorem \ref{thm:bound_restated}]
We proceed by induction on $n$, the number of vertices. In the base case of $n = 2$, the inequality \eqref{eq:upper_bound_f} is implied by Lemma \ref{lem:n=2}. For the inductive step we assume that \eqref{eq:upper_bound_f} is true for all graphs with $n - 1$ vertices. First, note that by \eqref{eq:astruct_system}, for any arithmetical structure on $G$ with $r_1 \geq r_2 \geq \cdots \geq r_n$, we have 
\begin{align} \label{eq:upper_bound_di}
d_1 = \frac{r_2}{r_1} \delta_{12} + \frac{r_3}{r_1}\delta_{13} + \cdots + \frac{r_n}{r_1} \delta_{1n} \leq \delta_{12} + \delta_{13} + \cdots + \delta_{1n} \leq \E{G}.
\end{align} 
In general, if $i$ is an index where $r_i \geq r_j$ for all $j \neq i$, then the above argument also shows $d_i \leq \E{G}$. Next we make the following observation. Fix a vertex $v_i$ and also fix a prescribed value for $d_i$ for an arithmetical structure. Once we have fixed these values, the graph $G'$ referenced in Construction \ref{con:recursive} is fixed. We claim there is at most one arithmetical structure on $G$ which satisfies these criteria which reduced to any given arithmetical structure on $G'$. To see this, let $(r_1, r_2, \ldots, r_n)$ and $(r_1', r_2', \ldots, r_n')$ be two arithmetical structures satisfying these criteria. Assume that $(r_1/g, r_2/g, \ldots, r_n/g) = (r_1'/g', r_2'/g', \ldots, r_n'/g')$, where $g$ and $g'$ are the $\gcd$ of $r_1, r_2, \ldots, r_n$ and $r_1', r_2', \ldots, r_n'$, and $r_i$ and $r_i'$ have now been removed. So, we have
$$d_ir_i = g \sum_{j \neq i} \frac{r_j}{g}$$
$$d_ir_i' = g' \sum_{j \neq i} \frac{r_j}{g'}$$
and hence we have $r_ig' = r_i'g$. By definition, $\gcd(g, r_i) = 1$, so $r_i \mid r_i'$ and $r_i' \mid r_i$, and $g = g'$. So the two arithmetical structures on $G$ are equal. 

This claim lets us bound the number of arithmetical structures on $G$ as follows. For each vertex $v_i$, we count the number of arithmetical structures $(r_1, r_2, \ldots, r_n)$ where the maximum of the values of $\textbf{r}$ is $r_i$. By \eqref{eq:upper_bound_di} we have $d_i \leq \E{G}$, so we get
$$\#A(G) \leq \sum_{i=1}^n\sum_{d_i = 1}^{\E{G}} \#A(G(v_1, d_i)).$$
By the same argument as in the proof of Theorem \ref{thm:largest_r1}, the number of edges in $G_{d_i}'$ is bounded by $\E{G}^2$. So by induction, we have
$$\#A(G) \leq \sum_{i=1}^n\sum_{d_i = 1}^{\E{G}} \frac{(n-1)!}{2} \E{G}^{2^{n-2} - 2} f\left(\E{G}^{2^{n-1}}\right)$$
$$= \frac{n!}{2} \E{G}^{2^{n-2} - 1} f\left(\E{G}^{2^{n-1}}\right),$$
completing the proof.

\end{proof}

\begin{remark}
\label{rem:paths_comp}
	If more is known about the structure of the graph $G$ then a much more accurate bound for $\#A(G)$ may be possible. As an extreme case, consider the path $G = P_{n+1}$ (see Example \ref{ex:paths}), where we have the exact count in terms of a Catalan number, $\#A(P_{n+1}) = C_n$ \cite[Theorem 3]{GlassEtAl}. This grows asymptotically as $4^n/(n^{3/2}\sqrt{\pi})$. On the other hand, the bound of Theorem \ref{thm:bound} with $n+1$ vertices and $\E{G} = n$ includes doubly exponential terms roughly of the form $n^{2^n}$, massively outpacing the Catalan numbers. Such a disparity in this case is not so surprising, given that a path has the minimal number of edges for a connected graph on $n$ vertices and that in our result we assume nothing about $G$ beyond the number of vertices and edges. Furthermore, in the bound in our result, the number of vertices is the variable that has the much bigger impact on the growth.
	
We will further discuss this disparity between Theorem \ref{thm:bound} and the true value of $\#A(G)$ in Section \ref{sec:Egyptian} in the case where $G=mK_n$ (see Remark \ref{rem:mKn_growth}).
\end{remark}

\section{Arithmetical structures on $mK_n$}
\label{sec:Egyptian}

\subsection{Specializing to $G = mK_n$} 

Consider the special case of $G = mK_n$, as in Example \ref{ex:mKn}. Recall that $mK_n$ is the graph on $n$ vertices with $\delta_{ij} = m$ for all $i \neq j$. Let $A_{\dec} (mK_n)$ denote the subset of arithmetical structures $\r \in A(mK_n)$ such that $r_i \geq r_{i+1}$ for $1 \leq i < n$. Using the same proof strategy as that of Theorem \ref{thm:bound_restated}, we can exploit the regularity of $mK_n$ to give a refinement.

\begin{corollary}
\label{cor:bound_mKn}
	Let $mK_n$ and $A_{\dec}(mK_n)$ be as defined above. Then
	\begin{equation}\label{eq:bound_mKn}
	\#A_{\dec}(mK_n) \leq \frac{(n-1)!}{2} \left( \prod_{k=0}^{n-4} (n-k)^{2^{n-3-k}-1} \right) \left(m^{2^{n-2}-1}\right)  \left(f \left( m^{2^{n-1}} \prod_{k=3}^n k^{2^{k-2}} \right) + 1 \right),
	\end{equation}
	where $f$ is any monotonically increasing function that is an upper bound for $\sigma_0$. In particular, we may again take $f(x) = x^{\frac{1.538 \log(2)}{\log(\log(x))}}$ as above.
\end{corollary}

\begin{proof}
	The proof follows that of Theorem \ref{thm:bound_restated}, proceeding by induction on $n$. The base case of $n=2$ is again a consequence of Lemma \ref{lem:n=2}, since $\#A_{\dec}(mK_2) = \frac{\#A(mK_2) + 1}{2}$. Assume \eqref{eq:bound_mKn} holds for $mK_{n-1}$. The key improvement to the argument in the proof of Theorem \ref{thm:bound_restated} is that we can refine \eqref{eq:upper_bound_di} since $\delta_{1i} = m$ for all $2 \leq i \leq n$:
	\[d_1 = m\left( \frac{r_2}{r_1} + \cdots + \frac{r_n}{r_1} \right) \leq (n-1)m.\]
	After removing $v_1$, the same argument as in the proof of Theorem \ref{thm:bound_restated} gives
		\[\#A_{\dec}(mK_n) \leq \sum_{d_1 = 1}^{(n-1)m} \#A_{\dec}\left((m^2 + d_1m)K_{n-1}\right).\]
	Since the upper bound in \eqref{eq:bound_mKn} is monotonic in $m$ for fixed $n$, we can safely bound $\#A_{\dec}(mK_n)$ above by $(n-1)m$ times the upper bound for $\#A_{\dec}\left((nm^2)K_{n-1}\right)$ as follows:
	\begin{align*}
		\#A_{\dec}(mK_n) & \leq (n-1)m \cdot \frac{(n-2)!}{2} \left(\prod_{k=0}^{n-5} (n-1-k)^{2^{n-4-k}-1} \right) \left((nm^2)^{2^{n-3}-1} \right) \left(f\left((nm^2)^{2^{n-2}} \prod_{k=3}^{n-1} k^{2^{k-2}} \right) + 1 \right)\\
		&= \frac{(n-1)!}{2} \left( \prod_{k=1}^{n-4} (n-k)^{2^{n-3-k}-1} \right) \left(n^{2^{n-3}-1}\right) \left(m^{1 + 2(2^{n-3}-1)} \right) \left(f \left( m^{2(2^{n-2})}n^{2^{n-2}}\prod_{k=3}^{n-1} k^{2^{k-2}} \right) + 1 \right)\\
		&= \frac{(n-1)!}{2} \left(\prod_{k=0}^{n-4}(n-k)^{2^{n-3-k}-1}\right)\left(m^{2^{n-2}-1}\right) \left(f\left(m^{2^{n-1}} \prod_{k=3}^{n} k^{2^{k-2}} \right) + 1\right)
	\end{align*}
	By induction this bound holds for all $n \geq 2$ and $m \geq 1$.
\end{proof}

\begin{remark}
\label{rem:bound_Kn}
	For any fixed $n$, Corollary \ref{cor:bound_mKn} improves on Theorem \ref{thm:bound} by a factor of a constant depending on $n$, since $\#E(mK_n) = m\binom{n}{2}$. This is a substantial improvement if we hold $m$ fixed and $n$ is allowed to vary. Asymptotically however, this bound can be improved upon (see Corollary \ref{cor:egypt_frac_bounds}) using results on Egyptian fractions of Browning--Elsholtz \cite{BrowningElsholtz} and Elsholtz--Planitzer \cite{ElsholtzPlanitzer}, which we discuss in  Subsection \ref{subsec:conn_Egy}.
\end{remark}

To see how the bound of Corollary \ref{cor:bound_mKn} compares to reality, we can use Construction \ref{con:recursive} to enumerate all the arithmetical structures on $mK_n$ for several small values of $m$ and $n$. This also serves as a proof of concept for how one might use the construction to produce an algorithm to generate all arithmetical structures on a given graph more generally.

Let $(r_1, r_2, r_3)$ be a candidate for an arithmetical structure on $mK_3$, and assume that $r_1 \geq r_2 \geq r_3$. Then by Lemma \ref{lem:recursive} and Lemma \ref{lem:n=2}, a necessary condition to be an arithmetical structure is that $r_2,r_3 \mid (m^2 + d_1m)$. Assuming this, we fix an integer $d_1$, which by \eqref{eq:upper_bound_di} is no more than $2m$. Next, we can check all possible pairs of $r_2$ and $r_3$ satisfying the divisibility above, and verify whether the corresponding $r_1 = m\cdot(r_2 + r_3)/d_1$ forms an arithmetical structure $(r_1, r_2, r_3)$. The following conditions are necessary and sufficient for this to occur.
\begin{enumerate}
  \item $r_1 \geq r_2$ which is equivalent to $(r_2 + r_3)\cdot m/d_1 \geq r_2$, or $mr_3 \geq (d_1 - m)r_2$. 
  \item Since $r_1 \in \Z$, $d_1 \mid m(r_2 + r_3)$. 
  \item By construction, $d_1 \in \Z$, but we also need that $d_2, d_3 \in \Z$. This is the same as
  $$r_2 \mid mr_3 + m\cdot\frac{m}{d_1}(r_2 + r_3).$$
  The same is true with the roles of $r_2$ and $r_3$ reversed. 
  \item $\gcd(r_1, r_2, r_3) = 1$. If $\gcd(r_2, r_3) = 1$ this is automatically satisfied. 
\end{enumerate}

Using Lemma \ref{lem:n=2}, we can enumerate all the arithmetical structures on $(m^2 + d_1m)K_2$ for $d_1 = 1, 2, \ldots, 2m$ and use conditions (1) --- (4) above to determine which lift to arithmetical structures on $mK_3$. We have written code in Magma that implements this algorithm to enumerate $A_{\dec}(mK_n)$, which can be found at \href{https://www.dropbox.com/s/zpratrd5mpiiu3v/arithmeticalStructures.m?dl=0}{this webpage}. For $n > 3$, some extra steps need to be performed.

We can translate the conditions above for general $n$ to recursively compute arithmetical structures on $mK_n$. Since $(r_2, r_3, \ldots, r_n)$ may have a common factor, we need to include an extra step in the algorithm. We recursively compute all possible arithmetical structures on $(m^2 + d_1m)K_{n-1}$ for each $d_1$. However, we allow this function to return values of $(r_2, r_3, \ldots, r_n)$ with a common factor, but would otherwise be an arithmetical structure. We use this to generate a list of $(r_1, r_2, \ldots, r_n)$ satisfying conditions $1$ through $3$ above, and then at the end we check which of these satisfy condition $4$. 

Table \ref{tab:small_mKn} compares $\#A_{\dec}(mK_n)$, enumerated by the methods described above, and the upper bound given by the floor of the right hand side of \eqref{eq:bound_mKn} for several small values of $n$ and $m$. The comparison shows that the bound of Corollary \ref{cor:bound_mKn} leaves much room for potential improvement, with its growth quickly outpacing the true value. The listed values were able to be computed reasonably quickly, but generating the full set of arithmetical structures on $mK_n$ becomes computationally challenging even for small $m$ values when $n > 3$.

\begin{center}
\begin{table}[H]
\begin{tabular}{ | c || c | c || c | c || c | c |}
  \hline
  \multirow{2}{*}{$m$} & \multicolumn{2}{c||}{$n=3$} & \multicolumn{2}{c||}{$n=4$} & \multicolumn{2}{c|}{$n=5$}\\ \cline{2-7} 
  & $\#A_{\dec}(mK_3)$ & Cor. \ref{cor:bound_mKn} & $\#A_{\dec}(mK_4)$ & Cor. \ref{cor:bound_mKn} &  $\#A_{\dec}(mK_5)$ & Cor. \ref{cor:bound_mKn} \\ \hline
  1 & 3 & 20 & 14 & 688 & 147 & 8567815 \\ \cline{6-7}
  2 & 10 & 56 & 108 & 23028 \\
  3 & 21 & 127 & 339 & 173664 \\ 
  4 & 28 & 229 & 694 & 717812 \\
  5 & 36 & 362& 1104 & 2141953\\
  6 & 57 & 526 & 1816 & 5209709\\
  7 & 42 & 720 & 2021 & 11012969\\
  8 & 70 & 946 & 3363 & 21019441\\
  9 & 79 & 1201 & 4053 & 37117341\\
  10 & 96 & 1487 & 5370 & 61657730\\ \cline{4-5} 
  100 & 1106 & 142796 \\
  101 & 164 & 145584 \\ \cline{1-3}
\end{tabular}
\caption{\label{tab:small_mKn} A comparison of the value $\#A_{\dec}(mK_n)$ with the bound given in Corollary \ref{cor:bound_mKn}.}
\end{table}
\end{center}

\begin{remark}[Growth of $\#A_{\dec}(mK_n)$]
\label{rem:mKn_growth}
	
	

This example will help illustrate why there is so much room for improvement. Given $m_1, n$, and an arithmetical structure on $m_1K_n$, Construction \ref{con:recursive} gives us a way to produce an associated arithmetical structure on $m_2K_{n-1}$, where $m_2 = m_1^2 + s_1m_1$ for the appropriate value of $s_1$. Iterating this procedure, we can get an arithmetical structure on $m_iK_{n - i + 1}$, where $m_i = m_{i-1}^2 + s_{i-1}m_{i-1}$ for $2 \leq i \leq n - 1$. For each such $i$, we therefore have that $m_i \geq m_{i-1}^2 + m_{i-1}$. Using this inequality for each $i$, the result of this process is an arithmetical structure on $m_{n-1}K_2$ with $m_{n-1} \geq f^{n-2}(m_1) = O(m_1^{2^{n-2}})$, where $f(m) = m^2 + m$. Furthermore, this is only for a given arithmetical structure. If we are interested in generating a list of all arithmetical structures on $m_1K_n$, there are many choices for the value of $s_1$ at each step. If every arithmetical structure on each of these $m_{n-1}K_2$ came from an arithmetical structure on $m_1K_n$ following this iterative procedure, then the comparison in Table \ref{tab:small_mKn} would be significantly closer. This is clearly not the case, and one challenge in improving the bounds given in Theorem \ref{thm:bound} and Corollary \ref{cor:bound_mKn} is that it is difficult to say which of which of these arithmetical structure on the base graphs such as $m_{n-1}K_2$ ``lift'' to arithmetical structures on the original graph, such as $m_1K_n$.

\end{remark}

\subsection{Connections to Egyptian fractions} 
\label{subsec:conn_Egy}

The study of Egyptian fractions focuses on integer solutions to \eqref{eq:egypt_frac}, for any given positive integers $m$ and $n$. Such solutions are in one-to-one correspondence with arithmetical structures on the graph $mK_n$. This allows us to use the theory of Egyptian fractions to study arithmetical structures on $mK_n$. While this correspondence is well known in the $m = 1$ case, we have not encountered this in the literature for general $m$, so we provide an elementary proof.

\begin{theorem}\label{thm:egypt_complete}
	The set $A(mK_n)$ is in one to one correspondence with solutions $(x_1, ..., x_n)$ to \eqref{eq:egypt_frac}. Explicitly, the arithmetical structure $(\r, \d) \in A(mK_n)$ corresponds to the solution $(d_1 + m, \ldots, d_n + m)$ to \eqref{eq:egypt_frac}.
\end{theorem}

\begin{proof}
	Let $(\textbf{r}, \textbf{d})$ be an arithmetical structure on $mK_n$ and recognize that
	\[\sum_{i=1}^n \frac{r_i}{m\sum_{j=1}^n r_j} = \frac{1}{m}.\]
	Using the system \eqref{eq:astruct_system} we may write
	\[\frac{r_i}{m\sum_{j=1}^n r_j} = \frac{r_i}{mr_i + d_ir_i} = \frac{1}{m + d_i}.\]
	Thus we have
	\[\sum_{i=1}^n \frac{1}{m + d_i} = \frac{1}{m},\] so by taking $x_i = m + d_i$ we have a solution to \eqref{eq:egypt_frac}. 
	
	We now show that given a solution \textbf{x} to \eqref{eq:egypt_frac}, we can find an arithmetical structure for which $x_i = m + d_i$. Setting $d_i = x_i - m$ in the system \eqref{eq:astruct_system}, we need the null space of 
	\begin{equation}\label{eq:egypt_sys}\begin{pmatrix} m - x_1 & m & \cdots & m\\ m & m-x_2 & \cdots & m\\ \vdots & \vdots & \ddots & \vdots\\m & m & \cdots & m - x_n \end{pmatrix}\end{equation} to have dimension exactly one. Subtracting the first row from all other rows, and scaling row $i$ by $1/x_i$ for $i \geq 2$, we have that this matrix is row equivalent to
	\[\begin{pmatrix} m - x_1 & m & m & \cdots & m\\ -x_1/x_2 & 1 & 0 & \cdots & 0\\-x_1/x_3 & 0 & 1 & \cdots & 0\\ \vdots & \vdots & \vdots & \ddots & \vdots\\-x_1/x_n & 0 & 0 & \cdots & 1 \end{pmatrix}\]
	After subtracting the first row by multiples of $m$ of the other rows, all entries are zero except the top left, which becomes \[m - x_1 + \frac{mx_1}{x_2} + \cdots \frac{mx_1}{x_n}.\] Multiplying this expression by $x_2 \cdots x_n$ gives $m(x_1 \cdots x_{n-1} + \cdots + x_2 \cdots x_n) - x_1\cdots x_n$, which is 0 by \eqref{eq:egypt_frac}. Hence the matrix is reduced to 
	\[\begin{pmatrix} 0 & 0 & 0 & \cdots & 0\\ -x_1/x_2 & 1 & 0 & \cdots & 0\\-x_1/x_3 & 0 & 1 & \cdots & 0\\ \vdots & \vdots & \vdots & \ddots & \vdots\\-x_1/x_n & 0 & 0 & \cdots & 1 \end{pmatrix},\]
	which clearly has rank $n-1$ and null space with dimension 1. An integral generator of the null space is
	\[\textbf{q} = (x_2\cdots x_n, x_1 x_3 \cdots x_n, \ldots, x_1 \cdots x_{n-1})^T.\]
	
	We construct an arithmetical structure by taking $\r = \textbf{q}/\gcd(\textbf{q})$ and setting $d_i = x_i - m$. These two processes, going from arithmetical structure on $mK_n$ to a solution $\textbf{x}$ to \eqref{eq:egypt_frac}, are clearly inverse to one another. 
\end{proof}

With Theorem \ref{thm:egypt_complete}, we can use known results bounding the number of Egyptian fraction representations for a given fraction to give a bound for $\#A_{\dec}(mK_n)$ and compare this to that of Corollary \ref{cor:bound_mKn}. Modifying slightly the notation of \cite{BrowningElsholtz}, we define
\begin{equation}\label{def:f_count}
	f_n(a,m) = \#\left\lbrace (x_1, \ldots, x_n) \in \N^n : x_1 \leq \cdots \leq x_n \text{ and } \frac{a}{m} = \frac{1}{x_1} + \cdots + \frac{1}{x_n} \right\rbrace
\end{equation}
to count the number of Egyptian fraction representations of $a/m$ by $n$ terms. Observe that an arithmetical structure $(\r,\d)$ on $mK_n$ satisfies $r_1 \geq \cdots \geq r_n$ if and only if $d_1 \leq \cdots \leq d_n$, so by the correspondence in the proof of Theorem \ref{thm:egypt_complete}, we have
\[f_n(1,m) = \#A_{\dec}(mK_n).\]
The best known asymptotic bounds for $f_n(1,m)$ are given by Elsholtz--Planitzer \cite[Theorems 1.1, 1.4]{ElsholtzPlanitzer}, improving on Browning--Elsholtz \cite[Theorems 2, 3]{BrowningElsholtz}, giving us the following corollary.

\begin{corollary}
\label{cor:egypt_frac_bounds}
	Let $n \geq 3$, $m \geq 1$, and fix $\epsilon > 0$. Then we have 
	\begin{align*}
		\#A_{\dec}(mK_3) &\ll_\epsilon m^{\frac{3}{5} + \epsilon},\\		
		\#A_{\dec}(mK_4) &\ll_\epsilon m^{\frac{28}{17} + \epsilon}, \text{ and}\\		
		\#A_{\dec}(mK_3) &\ll_\epsilon (nm)^\epsilon\left(n^{4/3}m^2\right)^{\frac{28}{17}2^{n-5}} \text{ when } n \geq 5.
	\end{align*}
\end{corollary}

Note that while this is an asymptotic improvement over Corollary \ref{cor:bound_mKn}, it does not give explicit constants. We also note that the exponential shape of the bounds in Corollaries \ref{cor:bound_mKn} and \ref{cor:egypt_frac_bounds} are somewhat similar. This may suggest that it would take a significant advance to close the large gap between the actual values and known bounds, as seen in Table \ref{tab:small_mKn}.

\bibliographystyle{alpha}
\bibliography{arith_structs_bib}

\newcommand{\etalchar}[1]{$^{#1}$}
\begin{thebibliography}{ABDL{\etalchar{+}}20}

\bibitem[ABDL{\etalchar{+}}20]{ArcherEtAl}
Kassie Archer, Abigail~C. Bishop, Alexander Diaz-Lopez, Luis~D.
  Garc\'{\i}a~Puente, Darren Glass, and Joel Louwsma.
\newblock Arithmetical structures on bidents.
\newblock {\em Discrete Math.}, 343(7):111850, 23, 2020.

\bibitem[BCC{\etalchar{+}}18]{GlassEtAl}
Benjamin Braun, Hugo Corrales, Scott Corry, Luis~David Garc\'{\i}a~Puente,
  Darren Glass, Nathan Kaplan, Jeremy~L. Martin, Gregg Musiker, and Carlos~E.
  Valencia.
\newblock Counting arithmetical structures on paths and cycles.
\newblock {\em Discrete Math.}, 341(10):2949--2963, 2018.

\bibitem[BE11]{BrowningElsholtz}
T.~D. Browning and C.~Elsholtz.
\newblock The number of representations of rationals as a sum of unit
  fractions.
\newblock {\em Illinois J. Math.}, 55(2):685--696 (2012), 2011.

\bibitem[Ble72]{Bleicher}
M.~N. Bleicher.
\newblock A new algorithm for the expansion of {E}gyptian fractions.
\newblock {\em J. Number Theory}, 4:342--382, 1972.

\bibitem[CV18]{CorralesValencia_clique_star}
Hugo Corrales and Carlos~E. Valencia.
\newblock Arithmetical structures on graphs.
\newblock {\em Linear Algebra Appl.}, 536:120--151, 2018.

\bibitem[EP20]{ElsholtzPlanitzer}
Christian Elsholtz and Stefan Planitzer.
\newblock The number of solutions of the erdős-straus equation and sums of k
  unit fractions.
\newblock {\em Proceedings of the Royal Society of Edinburgh: Section A
  Mathematics}, 150(3):1401–1427, 2020.

\bibitem[Guy04]{Guy}
Richard~K. Guy.
\newblock {\em Unsolved problems in number theory}.
\newblock Problem Books in Mathematics. Springer-Verlag, New York, third
  edition, 2004.

\bibitem[GW19]{GlassWagner}
Darren {Glass} and Joshua {Wagner}.
\newblock {Arithmetical Structures on Paths With a Doubled Edge}.
\newblock {\em arXiv e-prints}, page arXiv:1903.01398, Mar 2019.

\bibitem[HL20]{HarrisLouwsma}
Zachary Harris and Joel Louwsma.
\newblock On arithmetical structures on complete graphs.
\newblock {\em Involve}, 13(2):345--355, 2020.

\bibitem[Lor89]{Lorenzini}
Dino~J. Lorenzini.
\newblock Arithmetical graphs.
\newblock {\em Math. Ann.}, 285(3):481--501, 1989.

\bibitem[Nic88]{Nicolas}
Jean-Louis Nicolas.
\newblock On highly composite numbers.
\newblock In {\em Ramanujan revisited ({U}rbana-{C}hampaign, {I}ll., 1987)},
  pages 215--244. Academic Press, Boston, MA, 1988.

\end{thebibliography}

\end{document}